\documentclass{article}
\usepackage[left=30truemm, right=30truemm]{geometry}
\usepackage{amsmath,amssymb,amsthm}
\usepackage{bm}
\usepackage{mathtools}
\usepackage{xcolor}

\usepackage{enumerate}
\usepackage{enumitem}
\usepackage{float}

\usepackage{tikz}
\usepackage{caption}
\usepackage{subcaption}
\usetikzlibrary{arrows.meta,calc,decorations.markings,math,arrows.meta}

\usepackage[normalem]{ulem}

\definecolor{darkgreen}{RGB}{10,220,10}
\definecolor{rsred}{RGB}{255,100,255}

\usepackage[colorinlistoftodos,prependcaption,textsize=scriptsize, textwidth=20mm,]{todonotes}

\newtheorem{theorem}{Theorem}[]

\newtheorem{claim}{Claim}
\newtheorem{conjecture}[theorem]{Conjecture}
\newtheorem{observation}{Observation}

\usepackage[pdfauthor={derajan},pdfstartview=XYZ,bookmarks=true,
colorlinks=true,linkcolor=blue,urlcolor=magenta,citecolor=blue,pdftex,bookmarks=true,linktocpage=true,hyperindex=true]{hyperref}

\numberwithin{equation}{section}

\title{Results on proper conflict-free list coloring of graphs}
\author{
	Masaki Kashima\thanks{Faculty of Science and Technology, Keio University, Yokohama, Japan. Email: masaki.kashima10@gmail.com} \quad  
	Riste \v{S}krekovski\thanks{Faculty of Mathematics and Physics, University of Ljubljana, Faculty of Information Studies in Novo Mesto, and Rudolfovo - Science and Technology Centre Novo Mesto, Slovenia. Email:skrekovski@gmail.com} \quad 
	Rongxing Xu\thanks{School of Mathematical and Science, Zhejiang Normal University, Jinhua, China. Email:xurongxing@zjnu.edu.cn}
}

\begin{document}
\maketitle

\begin{abstract}
	Given a graph $G$ and a mapping $f:V(G) \to \mathbb{N}$, an $f$-list assignment of $G$ is a function that maps each $v \in V(G)$ to a set of at least $f(v)$ colors.  For an $f$-list assignment $L$ of a graph $G$, a proper conflict-free $L$-coloring of $G$ is a proper coloring $\phi$ of $G$ such that $\phi(v) \in L(v)$ for every vertex $v\in V(G)$ and $v$ has a color that appears precisely once at its neighborhood for every non-isolated vertex $v\in V(G)$. We say that $G$ is \emph{proper conflict-free $f$-choosable} if for any $f$-list assignment $L$ of $G$, there exists a proper conflict-free $L$-coloring of $G$. 
  For a non-negative integer $k$, we say that $G$ is \emph{proper conflict-free $({\rm degree}+k)$-choosable} if $G$ is proper conflict-free $f$-choosable where $f$ is a mapping with $f(v)= d_G(v)+k$ for every vertex $v\in V(G)$.
		 
  Motivated by degree-choosability of graphs, we investigate the proper conflict-free $({\rm degree}+k)$-choosability of graphs, especially for cases $k=1,2,3$.
	As the 5-cycle is not proper conflict-free $({\rm degree}+2)$-choosable and it is the only such graph we know, it is possible that every connected graph other than the 5-cycle is proper conflict-free $({\rm degree}+2)$-choosable and thus every graph is proper conflict-free $({\rm degree}+3)$-choosable.
	To support these, we show that every connected graph with maximum degree at most 3 distinct from the 5-cycle is proper conflict-free $(\text{degree}+2)$-choosable, and that $S(G)$ is proper conflict-free $(\text{degree}+2)$-choosable for every graph $G$, where $S(G)$ is a graph obtained from $G$ by subdividing each edge once.
  Furthermore, by adapting the technique of DP-colorings, we prove that every graph with maximum degree at most $4$ is proper conflict-free $({\rm degree}+3)$-choosable. 
\end{abstract}
	
	
\section{Introduction}\label{sec:intro}
Throughout this paper, we only consider simple, finite and undirected graphs.
We refer readers to Bondy and Murty~\cite{Bondy} for all terminologies and notations not defined here.

A $k$-coloring of a graph $G$ is a map $\phi: V(G)\to \{1,2,\ldots,k\}$.
A coloring $\phi$ is called \emph{proper} if every adjacent pair of vertices receive distinct colors.
The chromatic number of a graph $G$, denoted by $\chi(G)$, is the least positive integer $k$ such that $G$ admits a proper $k$-coloring.
	
For a graph $G$, a coloring $\phi$ is called \emph{proper conflict-free} if $\phi$ is a proper coloring of $G$, and for every non-isolated vertex $v\in V(G)$, there exists a color that appears exactly once in the neighborhood of $v$.
The \emph{proper conflict-free chromatic number} of a graph $G$, denoted by $\chi_{\textrm{pcf}}(G)$, is the least positive integer $k$ such that $G$ admits a proper conflict-free $k$-coloring.
The notion of proper conflict-free coloring of a graph was introduced by Fabrici, Lu\v{z}ar, Rindo\v{s}ov\'{a} and Sot\'{a}~\cite{Fabrici}, where they investigated proper conflict-free coloring of planar graphs and outerplanar graphs.
	
Another related concept is odd coloring of graphs, which was first introduced by Petru\v{s}evski and \v{S}krekovski~\cite{Petrusevski}.
A coloring $\phi$ of $G$ is called an \emph{odd coloring} of $G$ if $\phi$ is a proper coloring of $G$, and for every vertex $v\in V(G)$, there exists a color that appears odd times in the neighborhood of $v$.
The \emph{odd chromatic number} of a graph $G$, denoted by $\chi_o(G)$, is the least positive integer $k$ such that $G$ admits an odd $k$-coloring. By definitions, a proper conflict-free coloring of a graph $G$ is an odd coloring of $G$ as well, hence $\chi_o(G)\leq \chi_{\textrm{pcf}}(G)$ for any graph $G$.
	
In proper colorings of graphs, there is a trivial bound of the chromatic number that $\chi(G)\leq \Delta(G)+1$ for any graph $G$.
This bound does not work for proper conflict-free colorings or odd colorings of graphs since $C_5$, the 5-cycle, has the maximum degree 2 and $\chi_{\textrm{pcf}}(C_5)=\chi_o(C_5)=5$.
Caro, Petru\v{s}evski and \v{S}krekovski~\cite{Caro1} conjectured that $\chi_o(G)\leq \Delta(G)+1$ for every connected graph $G$ provided $\Delta(G) \geq 3$, and verified the case $\Delta(G) = 3$. They also conjectured the following stronger form  in \cite{Caro2}.
	
\begin{conjecture}\label{conj:pcf}
	For any connected graph $G$ with $\Delta(G) \geq 3$, $\chi_{\textup{pcf}}(G)\leq \Delta(G)+1$.
\end{conjecture}
	
Conjecture \ref{conj:pcf} for the case $\Delta(G)=3$ follows from an old result by Liu and Yu~\cite{LY2013}, where they studied the similar concept in the name of superliniar coloring, and their result also implies the list version. 
For general cases, the best known bound of proper conflict-free chromatic number in terms of the maximum degree is the following by Cho, Choi, Kwon and Park~\cite{CCKP2023}.

\begin{theorem}\label{thm:pcf}
  For any graph $G$, $\chi_{\textup{pcf}}(G)\leq 2\Delta(G)-1$.
\end{theorem}
    
In fact, they showed a result on proper $h$-conflict-free coloring of a graph, which is a coloring with more general setting.
For graphs with sufficient large maximum degree, it was proved recently by Liu and Reed~\cite{LRarxiv} that $\chi_{\textrm{pcf}}(G)\leq (1+o(1))\Delta(G)$ by using the probabilistic method.
There are some other results which describe asymptotic behaviors of proper conflict-free chromatic numbers and odd chromatic numbers (See \cite{Cranston,DOP2023,Kamyczura}).

We say a graph $G$ is \emph{degree-choosable} if $G$ is $f$-choosable with $f(v)=d_{G}(v)$ for each $v \in V(G)$.
It is well known that a connected graph $G$ is not degree-choosable if and only if $G$ is a \emph{Gallai-tree}, i.e. each block of $G$ is either a complete graph or an odd cycle. 
Degree-choosable graphs have been investigated in many papers \cite{Borodin,Borodin1979,Erdos,Thomassen,Vizing}. 
We consider the analogous concept of degree-choosability for proper conflict-free coloring.
We say that a graph $G$ is \emph{proper conflict-free $({\rm degree}+k)$-choosable} if $G$ is proper conflict-free $f$-choosable with $f(v)=d_{G}(v)+k$ for each $v \in V(G)$ and some constant $k$.
The following conjecture states the fundamental problem in this concept.
	
\begin{conjecture}\label{conj:general}
	There exists an absolute constant $k$ such that every graph is proper conflict-free $({\rm degree}+k)$-choosable.
\end{conjecture}

Since the 5-cycle is not proper conflict-free $4$-choosable, if Conjecture \ref{conj:general} has a positive answer, then the constant $k$ must be at least 3.
On the other hand, the 5-cycle is the only connected graph we know which is not proper conflict-free $({\rm degree}+2)$-choosable.
This naturally derives the following conjecture.

\begin{conjecture}\label{conj:except-5cycle}
  Every connected graph other than the 5-cycle is proper conflict-free $({\rm degree}+2)$-choosable.
\end{conjecture}

We cannot reduce $(\text{degree}+2)$ in Conjecture \ref{conj:except-5cycle} to $(\text{degree}+1)$ as it is shown by Caro, Petru\v{s}evski and \v{S}krekovski \cite{{Caro2}} that the cycle of length $\ell$ is not proper conflict-free $3$-colorable if $\ell$ is not divided by 3.
More strongly, there are many graphs which are not proper conflict-free $(\text{degree}+1)$-choosable in a sense.
We show the following.

\begin{theorem}\label{thm:deg 1}
  For any connected graph $H$, there is a connected graph $G$ such that $H$ is an induced subgraph of $G$ and $G$ is not proper conflict-free $(\text{degree}+1)$-choosable.
\end{theorem}

In particular, large maximum degree does not imply proper conflict-free $({\rm degree}+1)$-choosability of graphs. 
This is significantly different from the behavior of proper conflict-free $(\Delta+1)$-colorability conjectured in Conjecture \ref{conj:pcf}.
	
To support Conjecture \ref{conj:except-5cycle}, we verify it for two classes of graphs as follows.
For a graph $G$, the \emph{1-subdivision} of $G$, denoted by $S(G)$, is a graph obtained from $G$ by replacing each edge with a path of length 2.

\begin{theorem}\label{thm:G4 degree}
	If $G$ is a connected graph with maximum degree at most $3$ and $G$ is not isomorphic to $C_5$, then $G$ is proper conflict-free $({\rm degree}+2)$-choosable. 
\end{theorem}

\begin{theorem}\label{thm:1subdivided graph}
  For every graph $G$, $S(G)$ is proper conflict-free $({\rm degree}+2)$-choosable.
\end{theorem} 

If Conjecture \ref{conj:except-5cycle} has a positive answer, then obviously it implies that Conjecture \ref{conj:general}  has a positive answer by choosing $k=3$. 
Thus we consider proper conflict-free $({\rm degree}+3)$-choosability of graphs, and prove the following by using a technique of DP-coloring, which was introduced by Dvo\v{r}\'{a}k and Postle~\cite{Dvorak}.
	
\begin{theorem}\label{thm:main}
	Every graph $G$ with maximum degree at most $4$ is $({\rm degree}+3)$-choosable. In particular, every graph with maximum degree at most $4$ is proper conflict-free $7$-choosable.
\end{theorem}

Theorem \ref{thm:pcf} implies the bound $\chi_{\rm pcf}(G)\leq 5$ for graphs with maximum degree 3, and the bound $\chi_{\rm pcf}(G)\leq 7$ for graphs with maximum degree 4. Therefore, Theorems~\ref{thm:G4 degree} and \ref{thm:main} respectively extend this two results to list version.

The paper is organized as follows. 
We give a proof of Theorem \ref{thm:deg 1} in Section \ref{sec:pcf deg 1}, proofs of Theorems \ref{thm:G4 degree} and \ref{thm:1subdivided graph} in Section \ref{sec:pcf deg 2}, and a proof of Theorem \ref{thm:main} in Section \ref{sec:pcf deg 3}.
	
We end this section by introducing some notations. For a positive integer $k$, let $[k]$ denote the set of integers $\{0, 1, \dots , k-1\}$. Given a graph $G$ and a subgraph $H$ of $G$, for $v \in V(G)$ (not necessary in $H$), we denote $N_H(v)$ the neighbors of $v$ in $V(H)$. For a set $S \subseteq V(G)$, we denote  $N_H(S)$ all the vertices in $V(H)$ which have a neighbor in $S$. For a (partial) coloring $\phi$ of $G$ and a subset $R \subseteq V(G)$, let $\phi(R)$ be the set comprised of all colors used on $R$ of $\phi$ (not every vertex in $R$ needs to be colored in $\phi$). For a (partial) proper $h$-conflict-free coloring $\phi$ of $G$, let $\mathcal{U}_{\phi}(v,G)$ be the set of colors that appear exactly once at the neighborhood of $v$ in $G$, and let $\mathcal{U}_{\phi}(R,G)=\bigcup_{v \in R}\mathcal{U}_\phi(v,G)$.

\section{Degree plus 1 colors}\label{sec:pcf deg 1}

In this section, we give a proof of Theorem \ref{thm:deg 1}.

Let $H$ be an arbitrary connected graph and let $v_0$ be a vertex of $H$.
Let $d_H(v_0)=d$.
For each $i\in \{1,2,\dots , d+1\}$, let $C_i=v_0^iv_1^iv_2^iv_3^iv_0^i$ be a cycle of length $4$.
Let $G$ be a graph obtained from $H$ and $C_1, C_2, \dots , C_{d+1}$ by identifying $v_0$ and $v_0^1, v_0^2, \dots , v_0^{d+1}$ into a vertex $v$.
Then $H$ is an induced subgraph of $G$.
We define a list assignment $L$ of $G$ as follows.
\begin{itemize}
  \item for each $i\in \{1,2,\dots , d+1\}$, let $L(v_1^i)=L(v_2^i)=L(v_3^i)=\{3i-2, 3i-1, 3i\}$,
  \item $L(v)=\{1,2,\dots , 3d+3\}$, and
  \item for each vertex $u\in V(G)\setminus \{v\}$, let $L(u)=\{1,2,\dots , d_G(u)\}$.
\end{itemize}
Since $d_G(v)=d_H(v_0)+2(d+1)=d+2(d+1)=3d+2$, $L$ is a list assignment of $G$ such that $|L(u)|\geq d_G(u)+1$ for every vertex $u$ of $G$.
We show that $G$ is not proper conflict-free $L$-colorable.
Suppose that $G$ admits a proper conflict-free $L$-coloring $\phi$.
For each $i\in \{1,2,\dots , d+1\}$, since $\mathcal{U}_{\phi}(v_2^i,G)\neq \emptyset$, we have $\{\phi(v_1^i),\phi(v_2^i),\phi(v_3^i)\}=\{3i-2,3i-1,3i\}$.
In addition, since $\phi(v_j^i)\neq \phi(v)$ and $\mathcal{U}_{\phi}(v_j^i,G)\neq \emptyset$ for $j=1,3$, we have $\phi(v)\notin \{\phi(v_1^i),\phi(v_2^i),\phi(v_3^i)\}=\{3i-2,3i-1,3i\}$.
Combining these, we have $\phi(v)\notin \bigcup_{i=1}^{d+1}\{3i-2,3i-1,3i\}=\{1,2,\dots , 3d-3\}=L(v)$, a contradiction.

\section{Degree plus 2 colors}\label{sec:pcf deg 2}

In this section, we give proofs of Theorems \ref{thm:G4 degree} and \ref{thm:1subdivided graph}.

\subsection{Proof of Theorem \ref{thm:G4 degree}}
Let $L$ be a list assignment of $G$ with $|L(v)| \geq d_{G}(v)+2$ for each $v \in V(G)$. The statement is trivial when $G$ is isomorphic to $K_1$ or $K_2$, so we assume that $G$ has at least three vertices.
	
If $G$ is a cycle of length $\ell\leq 4$, then obviously $G$ admits a proper conflict-free $L$-coloring with $\ell$ distinct colors. We assume that $G=u_1u_2\cdots u_{\ell}u_1$ is a cycle for some $\ell \geq 6$. If $L(u_i)=L(u_j)$ for any $i \neq j$, then by the result of Caro et al.~\cite{Caro2}, $G$ is $L$-colorable. Thus without loss of generality, we assume that $L(u_1) \neq L(u_k)$. Let $\phi(u_1)$ be a color in $L(u_1)\setminus L(u_k)$. We choose colors for $\{u_2,u_3,\dots , u_{k-1}\}$ in the ascending order of indices such that $\phi(u_2)\in L(u_2)\setminus\{\phi(u_1)\}$ and $\phi(u_i)\in L(u_i)\setminus\{\phi(u_{i-1}), \phi(u_{i-2})\}$. Finally, we choose a color in $L(u_k)\setminus\{\phi(u_2),\phi(u_{k-2}),\phi(u_{k-1})\}$ as $\phi(u_{k-1})$ to obtain an $L$-coloring $\phi$ of $G$. It is easy to check $\phi$ is a proper conflict-free $L$-coloring of $G$.

Hence we may assume that $G$ is not a cycle, and thus $\Delta(G)=3$. 
Let $G_0, G_1, \dots , G_k$ be a sequence of subgraphs of $G$ such that $G_0=G$, $G_k$ is either isomorphic to $K_2$ or a graph with minimum degree at least 2, and $G_{i+1}=G_i-v_i$ for some $v_i\in V(G_i)$ with $d_{G_i}(v_i)=1$ for each $i\in\{0,1,\dots , k-1\}$. Note that $G_i$ is connected for every $i\in\{0,1,\dots , k\}$. 
Then for each $i \in \{0,1,\dots , k-1\}$, $G_i$ is proper conflict-free $L$-colorable if $G_{i+1}$ is proper conflict-free $L$-colorable. Indeed, suppose $\phi$ is a proper conflict-free $L$-coloring of $G_{i+1}$, and $x$ is the neighbor of $v_i$ in $G_i$. We choose a color $c_x\in \mathcal{U}_{\phi}(x,G_{i+1})$ arbitrarily, and choose a color in $L(v_i)\setminus\{\phi(x),c_x\}$ as $\phi(v_i)$ to extend $\phi$ to $G$. Obviously $\phi$ is a proper conflict-free $L$-coloring of $G_i$, as desired.
	
By the last paragraph, it suffices to show that $G_i$ is proper conflict-free $L$-colorable for some $i\in \{1,2,\dots ,k\}$. If $G_k$ is isomorphic to $K_2$, then it is proper conflict-free $L$-colorable. If $G_k$ has a vertex of degree 3, then by the result in Liu and Yu~\cite{LY2013}, $G_k$ is proper conflict-free $4$-choosable, and hence proper conflict-free $L$-colorable. Thus, we assume that $G_k=w_1w_2\cdots w_t w_1$ is a cycle of length $t \geq 3$. Without loss of generality, we assume $v_{k-1}$ is adjacent to $w_1$, and thus $d_G(w_1)=3$. Now we define an $L$-coloring $\phi$ of $G_{k-1}$ as follows: We choose colors $\phi(w_2)\in L(w_2)$, $\phi(w_3)\in L(w_3)\setminus\{\phi(w_2)\}$, and choose a color $\phi(w_j)$ for each of $\{w_4, w_5,\dots ,w_t\}$ in the ascending order of indices such that $\phi(w_j)\in L(w_j)\setminus\{\phi(w_{j-1}),\phi(w_{j-2})\}$. Finally we choose a color in $L(w_1)\setminus\{\phi(w_2), \phi(w_3), \phi(w_{t-1}), \phi(w_{t})\}$ as $\phi(w_1)$ and a color in $L(v_{k-1})\setminus\{\phi(w_1), \phi(w_2)\}$ as $\phi(v_{k-1})$. It is easy to check that $\phi$ is a proper conflict-free $L$-coloring of $G_{k-1}$.

\subsection{Proof of Theorem \ref{thm:1subdivided graph}}

Let $G$ be a graph with $n$ vertices $\{v_1, v_2,\dots , v_n\}$. The proof goes by induction on $n$. If $G$ is disconnected, then each component of $S(G)$ has a desired coloring and we simply combine them to obtain a desired coloring of $S(G)$. Thus we assume $G$ is connected. If $n=1$, then obviously $S(G)$ is isomorphic to $K_1$, which is proper conflict-free $(\text{degree}+2)$-choosable, so we assume $n\geq 2$. Let $S(G)$ be labeled as $V(S(G))=\{v_i\mid 1\leq i\leq n\}\cup \{u_e\mid e\in E(G)\}$ and $E(S(G))=\{v_iu_e\mid v_i\text{ incident with }e \text{ in }G\}$. Note that $d_{S(G)}(v_i)=d_G(v_i)$ for $i\in\{1,2,\ldots, n\}$ and $d_{S(G)}(u_e)=2$ for every $e\in E(G)$. Let $L$ be a list assignment of $G$ such that $|L(v_i)|\geq d_G(v_i)+2$ for each $i\in \{1,2,\dots ,n\}$ and $|L(u_e)|=4$ for each edge $e\in E(G)$. If $G$ is isomorphic to $K_2$, then we choose distinct colors $\phi(v_1)\in L(v_1)$ and $\phi(v_2)\in L(v_2)$, and choose a color in $L(u_{v_1v_2})\setminus \{\phi(v_1), \phi(v_2)\}$ as $\phi(u_{v_1v_2})$ to obtain a proper conflict-free $L$-coloring $\phi$ of $G$. Now we assume that $n\geq 3$ and the statement holds for every proper subgraph of $G$.
  
Suppose that $G$ contains a vertex of degree 1. Without loss of generality, we assume that $d_G(v_1)=1$ and let $e=v_1v_2$ be the unique edge incident with $v_1$ in $G$. We choose an arbitrary color $\alpha$ from $L(v_1)$. Let $G'=G-v_1$ and let $L'$ be a list assignment of $S(G')$ defined by $L'(v_2)=L(v_2)\setminus\{\alpha\}$ and $L'(w)=L(w)$ for every other vertex $w$. By induction hypothesis, $S(G')$ admits a proper conflict-free $L'$-coloring $\phi$. As $n\geq 3$ and $G$ is connected, $v_2$ is not an isolated vertex of $S(G')$, so we choose a color $c_2\in\mathcal{U}_{\phi}(v_2,S(G'))$ arbitrarily. Then we let $\phi(v_1)=\alpha$ and choose a color in $L(u_e)\setminus\{\phi(v_2),c_2,\alpha\}$ as $\phi(u_e)$ to extend $\phi$ to $S(G)$. It is easy to check that $\phi$ is a proper $L$-coloring of $S(G)$ and $\mathcal{U}_{\phi}(w,S(G))\neq \emptyset$ for every non-isolated vertex $w\in V(S(G'))\setminus\{v_2\}$. Since $\phi(u_e)\neq c_2$, the color $c_2$ is contained in $\mathcal{U}_{\phi}(v_2,S(G))$. Furthermore, $\alpha\in \mathcal{U}_{\phi}(u_e,S(G))$ since $\phi(v_2)\neq \alpha$, and $\phi(u_e)\in \mathcal{U}_{\phi}(v_1,S(G))$, these imply that $\phi$ is a proper conflict-free $L$-coloring. So we assume that $\delta(G) \geq 2$, and thus $\delta(S(G)) \geq 2$ as well.
  
If every vertex of $G$ has degree 2, and $L(u) = L(v)$ for any two distinct vertices $u,v$ in $S(G)$, then $S(G)$ is a cycle of length $2n$ and it is proper conflict-free $L$-colorable since every cycle other than the 5-cycle is proper conflict-free 4-colorable. Thus without loss of generality, we may assume that either $d_{S(G)}(v_1)=2$ and there is an edge $e$ incident to $v_1$ in $G$ satisfying $L(v_1)\neq L(u_e)$, or $d_{S(G)}(v_1)\geq 3$. Without loss of generality, we assume that $N_G(v_1)=\{v_2,v_3,\dots ,v_{d+1}\}$, where $d$ is the degree of $v_1$ in $G$. Let $e_i=v_1v_i\in E(G)$ for each $i\in \{2,3,\dots ,d+1\}$. For each color $c\in L(v_1)$, let $U_c=\{u_{e_i}\mid 2\leq i\leq d+1, c\in L(u_{e_i})\}$. Then there is a color $\alpha\in L(v_1)$ such that $|U_{\alpha}|\leq \min\{3,d-1\}$. Indeed, if $d=2$, then we may assume that $L(v_1)\neq L(u_{e_2})$, and thus the color $\alpha\in L(v_1)\setminus L(u_{e_2})$ satisfies $U_{\alpha}\subseteq \{u_{e_3}\}$. If $d=3$, then since $|L(v_1)|=5$, a color $\alpha\in L(v_1)\setminus L(u_{e_2})$ satisfies $U_{\alpha}\subseteq \{u_{e_3}, u_{e_4}\}$. If $d\geq 4$, then we have $|L(v_1)|\geq d+2$,  which implies  that there must be a color $\alpha\in L(v_1)$ such that $|U_{\alpha}|\leq 3$, for otherwise $4(d+2) \leq \sum_{c\in L(v_1)}|U_c|=\sum_{i=2}^{d+1} |L(u_{e_i})|=4d$, a contradiction. Without loss of generality, we may assume that $u_{e_2}\notin U_{\alpha}$.

Let $G'=G-(\{v_1\}\cup \{u_{e_i}\mid 2\leq i\leq d+1\})$ and let $L'$ be a list assignment of $S(G')$ such that $L'(v_i)=L(v_i)\setminus\{\alpha\}$ for each $i\in \{2,3,\dots ,d+1\}$ and $L'(w)=L(w)$ for every vertex $w\in V(S(G'))\setminus\{v_i\mid 2\leq i\leq d+1\}$. By induction hypothesis, $S(G')$ admits a proper conflict-free $L'$-coloring $\phi$. We extend $\phi$ to a proper conflict-free $L$-coloring of $G$. Note that $G'$ has no isolated vertices since $\delta(G)\geq 2$. For each $i\in \{2,3,\dots ,d+1\}$, let $c_i$ be a color in $\mathcal{U}_{\phi}(v_i, S(G'))$. Let $\phi(v_1)=\alpha$, and we choose $\phi(u_{e_i})\in L(u_{e_i})\setminus\{\phi(v_i), c_i, \alpha\}$ arbitrarily for each $u_{e_i}\in U_{\alpha}$. We define a color $\beta$ as follows: If $U_{\alpha}=\emptyset$, then let $\beta$ be a color in $L(u_{e_2})\setminus\{\phi(v_2), c_2, \alpha\}$. If $U_{\alpha}\neq\emptyset$ and $|\phi(U_{\alpha})|=1$, then since $\alpha\notin L(u_{e_2})$, let $\beta$ be a color in $L(u_{e_2})\setminus(\{\phi(v_2), c_2, \alpha\}\cup \phi(U_{\alpha}))$. Otherwise, i.e. $U_{\alpha}\neq\emptyset$ and $|\phi(U_{\alpha})|\geq 2$, the fact $|U_{\alpha}|\leq 3$ implies that some color appears exactly once at $U_{\alpha}$, so let $\beta$ be the color. Finally for each uncolored vertex $u_{e_i}$, since $u_{e_i}\notin U_{\alpha}$, we can choose a color $L(u_{e_i})\setminus\{\phi(v_j), c_j, \alpha, \beta\}$ as $\phi(u_{e_i})$. It is easy to verify that $\phi$ is a proper $L$-coloring of $S(G)$, and that $\mathcal{U}_{\phi}(w,S(G))=\mathcal{U}_{\phi}(w, S(G'))\neq\emptyset$ for each $w\in V(S(G'))\setminus\{v_i\mid 2\leq i\leq d+1\}$. Since $\phi(u_{e_i})\neq c_i$, $\mathcal{U}_{\phi}(v_i, S(G))$ contains the color $c_i$ for each $i\in \{2,3,\dots , d+1\}$. By definition of $L'$, we have $\phi(v_i)\neq\alpha=\phi(v_1)$, so $\alpha\in \mathcal{U}_{\phi}(u_{e_i}, S(G))$ for each $j\in \{2,3,\dots , d+1\}$. Furthermore, by the choice of colors $\{\phi(u_{e_i})\mid 2\leq i\leq d+1\}$, we have $\beta\in \mathcal{U}_{\phi}(v_1, S(G))$. Hence $\phi$ is a proper conflict-free $L$-coloring of $S(G)$. This completes the proof of Theorem \ref{thm:1subdivided graph}.

\vspace{\baselineskip}

We remark that $({\rm degree}+2)$ cannot be reduced to $({\rm degree}+1)$ in Theorems~\ref{thm:G4 degree} and \ref{thm:1subdivided graph}. Indeed, Caro, Petru\v{s}evski and \v{S}krekovski \cite{{Caro2}} observed that if $n$ is not divided by $3$, then $C_n$ is not proper conflict-free $3$-choosable. Clearly $S(C_{3k+1})=C_{6k+2}$ for every positive integer $k$, and thus $S(C_{3k+1})$ is not proper conflict-free $({\rm degree}+1)$-choosable.

\section{Degree plus 3 colors}\label{sec:pcf deg 3}

In this section, we give a proof of Theorem \ref{thm:main}.

We adapt the technique of the DP-coloring, which was introduced by Dvo\v{r}\'{a}k and Postle~\cite{Dvorak}.
Recall that $[k]$ denote the set of integers $\{0, 1, \dots , k-1\}$, and in this section, the indices are taken with modulo $k$ when it is related to the label of vertices.
	
Suppose that the statement is false, and let $G$ be a minimum counterexample in terms of the order, i.e. there is a list assignment $L$ of $G$ satisfying $|L(v)|\geq d_G(v)+3$ for each $v \in V(G)$ such that $G$ is not proper conflict-free $L$-colorable. Clearly, $G$ is connected and has at least 5 vertices as every vertex has at least 4 colors.

\begin{claim}\label{cl:smallvertex}
	The minimum degree of $G$ is at least 3.
\end{claim}
	
\begin{proof}
	If there is a vertex $v$ of degree $1$ in $G$ with the neighbor $u$, then by the minimality of $G$, there is a proper conflict-free $L$-coloring $\phi$ of $G-v$. We can extend $\phi$ to $G$ by setting $\phi(v) \in L(v)\setminus \{\phi(u),\alpha\}$ for $v$, where $\alpha \in \mathcal{U}_{\phi}(u,G)$. Thus have $\delta(G) \geq 2$ as $G$ is connected.
	
	Now we assume $v\in V(G)$ is a vertex of degree $2$ with neighbors $u_1$ and $u_2$.  Let $G'=(G-v)\cup\{u_1u_2\}$. As $\Delta(G') \leq 4$,  by the minimality of $G$, $G'$ admits a proper conflict-free $L$-coloring $\phi$. For each $i \in \{1,2\}$, we define $c_{u_i}$ as follows: If $|\mathcal{U}_{\phi}(u_i,G-v)|=0$, then all the neighbors of $u_i$ in $G-v$ receives the same color, say $\alpha_i$, in $\phi$, we let $c_{u_i}=\alpha_i$. Otherwise, $|\mathcal{U}_{\phi}(u_i,G-v)|\geq 1$, we choose $c_{u_i} \in \mathcal{U}_{\phi}(u_i,G-v)$. 
	We choose a color in $L(v)\setminus\{\phi(u_1), \phi(u_2), c_{u_1}, c_{u_2}\}$ as $\phi(v)$, which extends $\phi$ to $G$, a contradiction.
\end{proof}
	
Let $C=v_0v_1\cdots v_{k-1}v_0$ be a shortest cycle of $G$. Let $G'=G-V(C)$. For each $i\in [k]$, let $N_G(v_i)\setminus V(C)=\{u_i,w_i\}$ if $d_G(v_i)=4$, and let $N_G(v_i)\setminus V(C)=\{u_i\}$ if $d_G(v_i)=3$.
	
In the rest of the proof, we will first give an $L$-coloring (not necessarily proper conflict-free) $\phi$  of $G'$, and then we extend $\phi$ to a proper conflict-free $L$-coloring of $G$. For a coloring $\phi$ of $G'$,  a pair  of colors $(\alpha,\beta)$ and a vertex $v_i \in V(C)$, we say \emph{$v_i$ blocks $(\alpha,\beta)$ with respect to $\phi$} if either $d_{G}(v_i)=4$, $\alpha \neq \beta$ and $\phi(\{u_i,w_i\})=\{\alpha,\beta\}$, or $d_{G}(v_i)=3$ and $\phi(u_i)=\alpha = \beta$. For example, in the middle drawing of Fig.~\ref{fig-three}, $y$ blocks $(1,4)$ and $(4,1)$, $z$ blocks $(3,3)$. If $(\alpha,\beta)$ is not blocked by $v_i$, then we say $(\alpha,\beta)$ is \emph{feasible} for $v_i$. Observe that if $v_i$ blocks $(\alpha,\beta)$, then we can not extend $\phi$ to $G$ as $\phi(\{v_{i-1},v_{i+1}\})=\{\alpha,\beta\}$ because such an extension forces the vertex $v_i$ to have $\mathcal{U}_{\phi}(v_i,G)=\emptyset$.
By the definition, we have the following observation.
	
\begin{observation}\label{obs-block two}
	Every vertex blocks at most two color pairs. If $u$ blocks both $(\alpha_1,\beta_1)$ and $(\alpha_2,\beta_2)$, then they satisfy that $\alpha_1=\beta_2 \neq \alpha_2=\beta_1$.
\end{observation} 
\begin{claim}\label{cl:triangle}
	$k \geq 4$, i.e., $G$ contains no triangle.
\end{claim}
\begin{proof}
	Assume $k=3$. By the minimality of $G$, $G'$ admits a proper $L$-coloring $\phi$ such that every non-isolated vertex $v$ of $G'$ satisfies $\mathcal{U}_{\phi}(v,G')\neq \emptyset$. For each vertex $u \in N_{G'}(V(C))$ which is non-isolated in $G'$, we choose $c_u \in \mathcal{U}_{\phi}(u,G')$ arbitrarily.
	Let $L_C$ be a list assignment of $C$ defined by $L_C(v_i)=L(v_i)\setminus (\{c_u\mid u\in N_{G'}(v_i)\} \cup \phi(N_{G'}(v_i))$ for each $i\in [3]$, and observe that $|L_C(v_i)|\geq 3$. Now it suffices to show that $\phi$ can be extended to a proper $L$-coloring $\phi$ of $G$ such that $\phi(v_i)\in L_C(v_i)$ and $\mathcal{U}_\phi(v_i,G)\neq \emptyset$ for each $i\in [3]$. Indeed, if that is possible, then for each isolated vertex $x$ of $G'$, we have $|\mathcal{U}_{\phi}(x,G)|=3$, and for each non-isolated vertex $y$ of $G'$, we have $|\mathcal{U}_{\phi}(y,G)| > 0$ because we do not use $c_u$ in $V(C)$ for each $u \in N_{G'}(V(C))$.  
  
  If $L_C(v_0)=L_C(v_1)=L_C(v_2)$, then we extend $\phi$ to $G$ by giving an arbitrary $L_C$-coloring of $C$. Note that for each $i \in [3]$, $\phi(v_{i-1})\neq \phi(v_{i+1})$ and $\phi(\{v_{i-1}, v_{i+1}\})\cap \phi(N_{G'}(v_i))=\emptyset$,
  so $\mathcal{U}_{\phi}(v_i,G)\neq \emptyset$, a contradiction. 
		
	Thus without loss of generality, we may assume $L_C(v_0)\neq L_C(v_1)$, and that there exists four distinct colors $1,2,3,4$ such that $\{1,2\} \subsetneq L_C(v_0)$ and $\{3,4\} \subsetneq L_C(v_1)$. By Observation \ref{obs-block two}, $v_2$ blocks at most one of $(1,3),(1,4),(2,3),(2,4)$, so without loss of generality, we may assume that $(1,4),(2,3),(2,4)$ are all feasible for $v_2$. Let $\alpha$ be a color in $L_C(v_2)\setminus \{1,4\}$, then we deduce that $\phi(\{u_0,w_0\})=\{1,\alpha\}$ or $\phi(\{u_1,w_1\})=\{4,\alpha\}$. For otherwise
	we can extend $\phi$ by setting $\phi(v_0)=1$, $\phi(v_1)=4$ and $\phi(v_2)=\alpha$. Without loss of generality, we may also assume that $\phi(\{u_1,w_1\})=\{1,\alpha\}$, and this implies that $\alpha \neq 3$ as $3 \in L_C(v_1)$. If $\phi(\{u_0,w_0\})=\{4,\alpha\}$, then $\alpha \neq 2$, and thus we extend $\phi$ by setting $\phi(v_0)=2$, $\phi(v_1)=3$ and $\phi(v_2)=\alpha$, a contradiction. Thus we have $\phi(\{u_0,w_0\}) \neq \{4,\alpha\}$. If $\alpha \neq 2$, then we can extend $\phi$ by setting $\phi(v_0)=2$, $\phi(v_1)=4$ and $\phi(v_2)=\alpha$, a contradiction. Therefore, we have that $\alpha=2$. Let $\beta \in L_C(v_0) \setminus \{1,2\}$. Clearly, there exists a color $\gamma \in L_C(v_1) \setminus \{\beta\}$ such that $(\beta,\gamma)$ is feasible for $v_2$ and we know that $\gamma \neq 2$ as $2 \in \phi(\{u_1,w_1\})$, we then extend $\phi$ by setting $\phi(v_0)=\beta$, $\phi(v_1)=\gamma$, and $\phi(v_2)=2$.  See the leftmost drawing on Fig.~\ref{fig-three} for illustration.
\end{proof}
	
Now we observe that \emph{$G'$ has no isolated vertices.} Indeed, if $u$ is an isolated vertex of $G'$, then since $\delta(G)\geq 3$ by Claim \ref{cl:triangle}, there exist vertices $v_a, v_b, v_c\in V(C)\cap N_G(u)$ with $0\leq a<b<c\leq k-1$.
By Claim~\ref{cl:triangle}, $b-a \geq 2$ and $c-b \geq 2$. Thus $v_auv_cv_{c+1}\cdots v_{a-1}v_a$ is a shorter cycle than $C$, a contradiction.
	
Let $G''$ be the graph obtained from $G'$ by adding all the edges  $u_iw_i$ for each $i \in [k]$ if $w_i$ exists, and removing the multiple edges. 
Clearly, $\Delta(G'') \leq 4$, and by the minimality of $G$, $G''$ has a proper conflict-free $L$-coloring $\phi$ with $\phi(u_i)\neq \phi(w_i)$ for every $i\in [k]$ if $d_G(v_i)=4$. Note that for each $u\in N_{G}(V(C))$, $\mathcal{U}_{\phi}(u,G'')\neq \emptyset$, but it is possible that $\mathcal{U}_{\phi}(u,G')= \emptyset$ which holds only if $|\phi(N_{G'}(u))|=1$ (the vertex $x$ in the second drawing on  Fig.~\ref{fig-three} is a such example).  For each $u\in N_{G}(V(C))$, we define a color $c_u$ as follows (see again the second drawing on  Fig.~\ref{fig-three} for illustration): If $\mathcal{U}_{\phi}(u,G') \neq \emptyset$, then let $c_u$ be an arbitrarily color in it, otherwise, let $c_u$ be the unique color in $\phi(N_{G'}(u))$.
	
Let $L_C$ be a list assignment of $C$ with $L_C(v_i)=L(v_i)\setminus (\{c_u \mid u\in N_{G'}(v_i)\} \cup \phi(N_{G'}(v_i)))$ for each $i\in [k]$. Observe that $|L_C(v_i)|\geq 3$, moreover $|L_C(v_i)|\geq 4$ if $d_{G}(v_i)=3$. In the rest of the proof, we will extend $\phi$ to a proper $L$-coloring of $G$ by choosing a color from $L_C(v_i)$ for each $i \in [k]$. We observe that after doing this, for each $u \in V(G) \setminus V(C)$, there will be one color that appears exactly once at the neighborhood of $v$ in $G$. Indeed, this is clearly true if $u \in V(G) \setminus (V(C) \cup N_{G'}(C))$. Suppose $u\in N_{G'}(v_i)$ for some $i\in [k]$. If $\mathcal{U}_{\phi}(u,G') = \emptyset$, then the color we chose for $v_i$ appears exactly once at the neighborhood of $u$,  otherwise, $c_u$ does. Therefore, it suffices to give a proper $L_C$-coloring of $C$ such that in the resulting coloring of $G$, for each vertex in $C$, some color appears exactly once at its neighborhood.
	
	\begin{figure}[h]
		\begin{minipage}{0.29\textwidth} 
			\centering
			\begin{tikzpicture}[>=latex,	
				roundnode/.style={circle, draw=black,fill= white, minimum size=1.3mm, inner sep=0.1pt}]
				\node[roundnode]  (v1) at (30:0.7) {};
				\node[roundnode]  (v0) at (150:0.7) {};
				\node[roundnode]  (v2) at (270:0.7) {};
				
				\draw (v0)--(v1)--(v2)--(v0);
				
				\draw (v1) to (15:1.4);
				\draw (v1) to (0:1.2);
				
				\node at (15:1.6) {1};
				\node at (-5:1.35) {$\alpha$};
				
				\draw (v0) to (165:1.4);
				\draw (v0) to (180:1.2); 
				
				\draw [rounded corners] (-0.8,0.5) rectangle (-0.4,1.8);
				\node at (-0.6,0.7) {$1$};
				\node at (-0.6,1.1) {$2$};
				\node at (-0.6,1.5) {$\beta$};
				
				\node [rotate = -50] at (-1.5,0) {$\neq \{4,\alpha\}$};
				
				\draw [rounded corners] (0.4,0.5) rectangle (0.8,1.5);
				\node at (0.6,0.7) {$3$};
				\node at (0.6,1.1) {$4$};
				
				\draw[dashed] (-0.5,0.7) -- (0.5,0.7);
				
				\draw (v2) to (-100:1.4);
				\draw (v2) to (-80:1.4);
				
				\node at (30:0.3) {$v_1$};
				\node at (175:0.65) {$v_0$};
				\node at (250:0.75) {$v_2$};
				\node at (315:0.9) {$\alpha=2$};
			\end{tikzpicture} 
			\label{fig-triangle}
		\end{minipage}
		\begin{minipage}{0.34\textwidth} 
			\centering
			\begin{tikzpicture}[>=latex,	
				roundnode/.style={circle, draw=black,fill= white, minimum size=1.3mm, inner sep=0.1pt},
				bluenode/.style={circle, draw=black,fill=blue, minimum size=1.5mm, inner sep=0pt},
				rednode/.style={circle, draw=black,fill=red, minimum size=1.5mm, inner sep=1pt}, 
				greennode/.style={circle, draw=black,fill=darkgreen, minimum size=1.5mm, inner sep=1pt},
				violetnode/.style={circle, draw=black,fill=violet, minimum size=1.5mm, inner sep=1pt},
				yellownode/.style={circle, draw=black,fill=yellow, minimum size=1.5mm, inner sep=1pt}]
				
				\draw circle(0.7); 
				
				\node[roundnode]  (A1) at (-20:0.7) {};
				\node[rednode]  (A2) at (-20:1.4) {};
				\filldraw [draw=white, fill=gray!70, opacity=0.5] (1.32,-0.4) rectangle (1.52,0);
				\node at (-10:1.3) {$2~3$}; 
				\node [bluenode]  (A3) at (-10:2.1) {};
				\node at (-9:2.35) {1};
				\node [bluenode] (A4) at (-20:2.1) {};
				\node at (-19:2.35) {1};
				\node [greennode] (A5) at (-30:2.1) {};
				\node at (-30:2.35) {3};
				\foreach \t in {1,3,4,5}{
					\draw (A2)--(A\t);} 
				
				\node [roundnode]  (B1) at (30:0.7) {};
				\node [greennode]  (B2) at (30:1.4) {};
				\filldraw [draw=white, fill=gray!70, opacity=0.5] (1.35,0.3) rectangle (1.55,0.7);
				\node at (19:1.4) {$3~1$}; 
				
				\node  [bluenode]  (B3) at (35:2.1) {};
				\node  at (34:2.35) {1};
				
				\node [rednode] (B4) at (25:2.1) {}; 
				\node at (25:2.35) {2}; 
				\foreach \t in {1,3,4}{
					\draw (B2)--(B\t);} 
				
				\node[roundnode]  (C0) at (135:0.7) {};	
				\node [rednode]  (C1) at (125:1.4) {};
				\filldraw [draw=white, fill=gray!70, opacity=0.5] (-0.4,1) rectangle (-0.2,1.4);
				\node  at (110:1.3) {$2~4$};
				
				\node [bluenode]  (C2) at (145:1.4) {};
				\node at (138:1.6) {$x$};
				\filldraw [draw=white, fill=gray!70, opacity=0.5] (-1.18,0.3) rectangle (-0.98,0.7);
				\node at (158:1.3) {$1~3$};
				\node  [bluenode] (C3) at (110:2) {};
				\node at (109:2.25) {1};
				\node   [bluenode] (C4) at (118:2.1) {}; 
				\node at (117:2.35) {1};
				\node [yellownode] (C5) at (128:2.1) {};
				\node at (128:2.35) {4};
				\node [greennode]  (C6) at (140:2.1) {};
				\node  at (140:2.35) {3}; 
				\node [greennode]  (C7) at (150:2.1) {}; 
				\node  at (150:2.35) {3};
				\node [greennode]  (C8) at (158:2) {}; 
				\node  at (158:2.25) {3};
				\foreach \t in {0,3,4,5}{
					\draw (C1)--(C\t);}  
				\foreach \t in {0,6,7,8}{
					\draw (C2)--(C\t);}  
				\draw [dashed] (C1)--(C2);

				\node[roundnode]  (D0) at (210:0.7) {};	
				\node[bluenode]  (D1) at (200:1.4) {};
				\filldraw [draw=white, fill=gray!70, opacity=0.5] (-1.25,-0.4) rectangle (-1.05,0);
				\node at (190:1.3) {$1~5$};
				\node[yellownode]  (D2) at (220:1.4) {};
				\filldraw [draw=white, fill=gray!70, opacity=0.5] (-0.75,-1.2) rectangle (-0.55,-0.8);
				\node  at (233:1.3) {$4~1$};
				\node [greennode] (D3) at (190:2) {};
				\node at (187:2.25) {3};
				\node [violetnode]  (D4) at (200:2) {}; 
				\node at (200:2.25) {5}; 
				\node [rednode] (D5) at (215:2.1) {}; 
				\node at (215:2.35) {2}; 
				\node [rednode] (D6) at (225:2.1) {}; 
				\node at (225:2.35) {2}; 
				\foreach \t in {0,3,4}{
					\draw (D1)--(D\t);}  
				\foreach \t in {0,5,6}{
					\draw (D2)--(D\t);}  
				\draw (D1)--(D2);
				
				\node at  (90:0.45) {$C$};	
				
				\node at (210:0.45) {$y$};
				\node at (35:0.45) {$z$};
			\end{tikzpicture} 
		\end{minipage}
		\begin{minipage}{0.29\textwidth} 
			\centering
			\begin{tikzpicture}[>=latex,	
				roundnode/.style={circle, draw=black,fill= white, minimum size=1.3mm, inner sep=0.1pt}]
				\node[roundnode]  (v1) at (45:0.7) {};
				\node[roundnode]  (v0) at (135:0.7) {};
				\node[roundnode]  (v3) at (225:0.7) {};
				\node[roundnode]  (v2) at (315:0.7) {};
				
				\draw (v0)--(v1)--(v2)--(v3)--(v0);
				
				\draw (v1) to (20:1.4);
				\draw (v1) to (10:1.2);

				\draw (v0) to (165:1.4);
				\node at (168:1.6){2};
				\draw (v3) to (195:1.4);
				\node at (192:1.6){1};
				
				\draw [rounded corners] (-0.7,-0.6) rectangle (-0.3,-1.4);
				\node at (-0.5,-0.8) {$2$};  
				
				\draw [rounded corners] (-0.7,0.6) rectangle (-0.3,1.6);
				\node at (-0.5,0.8) {$1$};
				\node at (-0.5,1.2) {$\alpha$}; 
				\node at (-0.3,2){$\alpha \notin L_C(v_1)$};
				
				\draw [rounded corners] (0.3,0.6) rectangle (0.7,1.6);
				\node at (0.5,0.8) {$1$};
				\node at (0.5,1.2) {$2$};
				
				\draw (v2) to (-10:1.2);
				\draw (v2) to (-20:1.4);
				\draw [rounded corners] (0.3,-0.6) rectangle (0.7,-1.4);
				\node at (0.5,-0.8) {$1$};
				\node at (0.5,-1.2) {$2$};

				\node at (45:0.4) {$v_1$};
				\node at (162:0.75) {$v_0$};
				\node at (198:0.75) {$v_3$};
				\node at (315:0.4) {$v_2$};
			\end{tikzpicture}
			\label{fig-four cycle}
		\end{minipage}
		\caption{The left drawing is an illustration for Claim \ref{cl:triangle}. In the middle drawing, for each vertex $u$, a number without shading represents $\phi(u)$, and a number with shading represents the color chosen as $c_u$.
        The right drawing is an illustration for Claim \ref{cl:4-cycle}. }
		\label{fig-three}
	\end{figure}
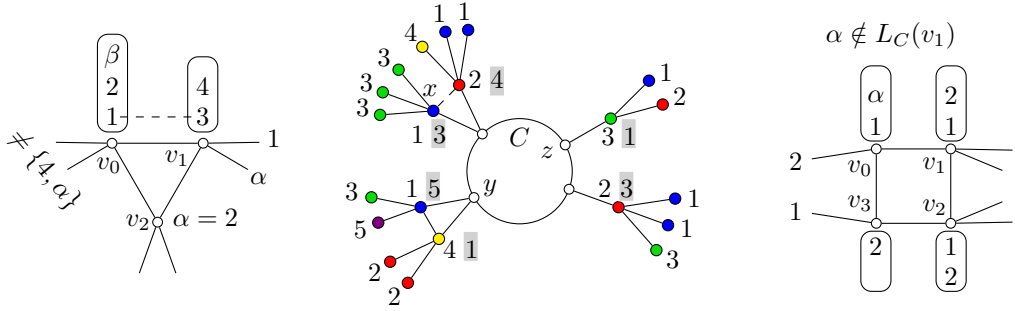
	 
Now we have the following two observations which we use.

\begin{observation}\label{obs-block}
	Let $C=xzywx$, $L'_C(x) \subseteq L_C(x)$ and $L'_C(y) \subseteq L_C(y)$ be two non-empty sets. If $|L'_C(x)|+|L'_C(y)|\geq 4$, then there exists a pair of colors in $L'_C(x) \times L'_C(y)$ which is feasible for both $z$ and $w$.
\end{observation}
\begin{proof}
	If $z$ blocks two color pairs $(\alpha_1,\beta_1),(\alpha_2,\beta_2) \in L'_C(x) \times L'_C(y)$, then by Observation \ref{obs-block two}, we have $\alpha_1=\beta_2 \neq \alpha_2=\beta_1$, and $\{\alpha_1,\alpha_2\} \subseteq  L'_C(x) \cap  L'_C(y)$. Thus $(\alpha_1,\alpha_1)$ or $(\alpha_2,\alpha_2)$ is not blocked by $w$. Hence by the symmetry of $z$ and $w$, we may assume that both $z$ and $w$ blocks at most one color pair in $L'_C(x) \times L'_C(y)$, but note that the condition $|L'_C(x)|+|L'_C(y)|\geq 4$ implies $|L'_C(x)|\geq 3$, or $|L'_C(y)|\geq 3$ or $\min\{|L'_C(x)|, |L'_C(y)|\}\geq 2$. In any case, $L'_C(x)\times L'_C(y)$ contains at least three pairs, thus there is a desired pair in $L'_C(x)\times L'_C(y)$.
\end{proof}
	
\begin{observation}\label{obs-extend}
	Let $C=xzywx$. If there is a color pair $(\alpha,\beta) \in L_C(x)\times L_C(y)$ which is feasible for both $z$ and $w$ (possibly, $\alpha=\beta$), and a color pair in  $L_C(z)\setminus \{\alpha,\beta\} \times L_C(w) \setminus \{\alpha,\beta\}$ which is feasible for both $x$ and $y$, then $\phi$ can be extended to a  proper conflict-free $L$-coloring of $G$.
\end{observation}
	
\begin{claim}\label{cl:4-cycle}
	$k\geq 5$, i.e., $G$ contains no 4-cycle.
\end{claim}
\begin{proof} 
	Suppose to the contrary, $k=4$. By Observation \ref{obs-block two} and the fact that $L_C(v_1) \times L_C(v_3)$ contains at least six pair of distinct colors, there must be a pair of distinct colors, say $(1,2) \in L_C(v_1) \times L_C(v_3)$, which is feasible for both $v_0$ and $v_2$. 
  By Observations~\ref{obs-block}, \ref{obs-extend} and the assumption that $|L_{C}(v_{i})| \geq 3$, we have $1 \leq |L_C(v_i)\setminus \{1,2\}| \leq 2$  for $i \in \{0,2\}$, 
  and at least one of $|L_C(v_0)\setminus \{1,2\}|$ and $|L_C(v_2)\setminus \{1,2\}|$ is equal to 1.
  Without loss of generality, we may assume that $1\in L_C(v_0)$ and $\{1,2\} \subsetneq L_C(v_2)$. We deduce that $d_{G}(v_3)=3$ and $\phi(u_3)=1$, for otherwise $(1,1) \in L_C(v_0) \times L_C(v_2)$ is feasible for both $v_1$ and $v_3$ and $|L_C(v_1)\setminus \{1\}|+|L_C(v_3)\setminus \{1\}| \geq 4$, a contradiction by Observation \ref{obs-extend}. Thus $1 \notin L_C(v_3)$. Observe that $(1,2)$ is feasible for both $v_1$ and $v_3$, this implies that $2 \in L_C(v_1)$ and $|L_C(v_1)| = 3$ again by Observation \ref{obs-extend}. So we deduce that $d_{G}(v_0)=3$ and $\phi(u_0)=2$ as $(2,2) \in L_C(v_1)\times L_C(v_3)$ must be blocked by $v_0$.
  Therefore, $2 \notin L_C(v_0)$, but recall that $2 \in L_C(v_1)$ and $|L_C(v_1)|=3$, so there exists a color $\alpha \in L_C(v_0)\setminus L_C(v_1)$. Then $(\alpha,1) \in L_C(v_0)\times L_C(v_2)$ is feasible for both $v_1$ and $v_3$, and $|L_C(v_1)\setminus \{\alpha,1\}|+|L_C(v_3)\setminus \{\alpha,1\}| \geq 2+2 \geq 4$, a contradiction by Observation \ref{obs-extend}. See the rightmost drawing in Fig.~\ref{fig-three} for an illustration.
\end{proof}
	
By Claims~\ref{cl:triangle} and \ref{cl:4-cycle}, $G$ has no triangles and $4$-cycles. Recall that by the assumption, $C$ is a shortest cycle, so we have $N_{G'}(v_i) \cap N_{G'}(v_j) = \emptyset$ for any distinct $i,j\in [k]$.

We construct an auxiliary graph $H$ with $V(H)=\bigcup_{i=1}^k X_i$ where $X_i=\{v_i\}\times L_C(v_i)$, and $E(H)= E_0 \cup E_1\cup E_2$, where 
\begin{align*}
	E_0 & =  \{(v_i,\alpha)(v_i,\beta)\mid \alpha,\beta \in L_C(v_i),\alpha \neq \beta,i\in [k]\},\\
	E_1 & =  \{(v_i, \alpha)(v_{i+1}, \alpha)\mid \alpha \in L_C(v_i)\cap L_C(v_{i+1}), i \in [k]\}, \text{and}\\
	E_2 & =  \{(v_i, \alpha)(v_{i+2}, \beta)\mid (\alpha,\beta) \in L_C(v_i)\times L_C(v_{i+2}) \text{ is blocked by } v_{i+1}, i \in [k]\}.
\end{align*} 
The edges in $E_1$ (resp. $E_2$) are called \emph{short} (resp. \emph{long}) edges. The edges incident with $(v_i,\alpha) \in X_i$ with other endpoints in $X_{i+1}\cup X_{i+2}$ (resp. $X_{i-1}\cup X_{i-2}$) are called \emph{forward} (resp. \emph{backward}) edges of $(v_i,\alpha)$. See Fig.~\ref{fig-H} for illustration.
	
	\begin{figure}[h]
		\centering  
		\begin{tikzpicture}[>=latex,
			roundnode/.style={circle, draw=black,fill=white, minimum size=1.5mm, inner sep=0pt},
			rednode/.style={circle, draw=black,fill=red, minimum size=1.5mm, inner sep=0.2pt},
			cyannode/.style={circle, draw=black,fill=cyan, minimum size=1.5mm, inner sep=0pt},
			greennode/.style={circle, draw=black,fill=darkgreen, minimum size=1.5mm, inner sep=0pt},
			violetnode/.style={circle, draw=black,fill=violet, minimum size=1.5mm, inner sep=0pt},
			orangenode/.style={circle, draw=black,fill=yellow, minimum size=1.5mm, inner sep=0pt},
			bluenode/.style={circle, draw=black,fill=blue, minimum size=1.5mm, inner sep=0pt}]
			
			\filldraw [draw=white, fill=gray!50, opacity=0.5, rounded corners] (1.75,1.8) rectangle (2.25,7.2);
			\filldraw [draw=white, fill=gray!50, opacity=0.5, rounded corners] (3.25,1.8) rectangle (3.75,7.2);
			\filldraw [draw=white, fill=gray!50, opacity=0.5, rounded corners] (4.75,1.8) rectangle (5.25,7.2);
			\filldraw [draw=white, fill=gray!50, opacity=0.5, rounded corners] (6.25,1.8) rectangle (6.75,7.2);
			\filldraw [draw=white, fill=gray!50, opacity=0.5, rounded corners] (7.75,1.8) rectangle (8.25,7.2);
			\filldraw [draw=white, fill=gray!50, opacity=0.5, rounded corners] (9.25,1.8) rectangle (9.75,7.2);
			
			\draw (2,8.5)--(9.5,8.5);
			\draw [dashed] (2,8.5)--(1,8.5);
			\draw [dashed] (9.5,8.5)--(10.5,8.5);

			\node [roundnode] (vl) at (1,8.5){};
			\node [roundnode] (v0) at (2,8.5){}; 
			\node [roundnode] (v1) at (3.5,8.5){};
			\node [roundnode] (v2) at (5,8.5){};
			\node [roundnode] (v3) at (6.5,8.5){};
			\node [roundnode] (v4) at (8,8.5){}; 
			\node [roundnode] (v5) at (9.5,8.5){};
			\node [roundnode] (vr) at (10.5,8.5){};
			
			\draw (vl) ..controls (0,10.3) and (11.5,10.3).. (vr);
			\draw (vl) -- (0.2,8.3);
			\draw (vr) -- (11.3,8.3);
			\draw (vr) -- (11.3,8.8);
			
			\filldraw [draw=white, fill=red] (0,6.85) rectangle (0.6,7.15);`
			\filldraw [draw=white, fill=cyan] (0,5.85) rectangle (0.6,6.15);
			\filldraw [draw=white, fill=darkgreen] (0,4.85) rectangle (0.6,5.15);
			\filldraw [draw=white, fill=violet] (0,3.85) rectangle (0.6,4.15);
			\filldraw [draw=white, fill=yellow] (0,2.85) rectangle (0.6,3.15);
			\filldraw [draw=white, fill=blue] (0,1.85) rectangle (0.6,2.15);
			
			\node at (-0.4,7) {$1$};
			\node at (-0.4,6) {$2$};
			\node at (-0.4,5) {$3$};
			\node at (-0.4,4) {$4$};
			\node at (-0.4,3) {$5$};
			\node at (-0.4,2) {$6$};
			
			\node [rednode] (x01) at (2,7){}; 
			\node [greennode] (x02) at (2,5){};
			\node [violetnode] (x03) at (2,4){};
			\node [orangenode] (x04) at (2,3){};
			
			\node [rednode] (x11) at (3.5,7){}; 
			\node [cyannode] (x12) at (3.5,6){};
			\node [greennode] (x13) at (3.5,5){};
			\node [bluenode] (x14) at (3.5,2){};
			
			\node [rednode] (x21) at (5,7){};
			\node [violetnode] (x22) at (5,4){};
			\node [orangenode] (x23) at (5,3){};
			\node [bluenode] (x24) at (5,2){};
			
			\node [greennode] (x31) at (6.5,5){};
			\node [violetnode] (x32) at (6.5,4){};
			\node [bluenode] (x33) at (6.5,2){};
			
			\node [rednode] (x41) at (8,7){};
			\node [cyannode] (x42) at (8,6){};
			\node [orangenode] (x43) at (8,3){};
			
			\node [rednode] (x51) at (9.5,7){};
			\node [cyannode] (x52) at (9.5,6){};
			\node [orangenode] (x53) at (9.5,3){};
			\node [bluenode] (x54) at (9.5,2){};
			
			\draw (1.75,9)--(v0)--(2.25,9);
			\node at (1.75,9.2) {$2$};
			\node at (2.25,9.2) {$6$};
			
			\draw (x01)--(x11)--(x21);
			\draw (x01)--(x11)--(x21);
			\draw (x02)--(x13);
			\draw [line width=1.2pt](x13)--(x31)--(x54);
			\draw [line width=1.2pt] (x03)--(x23)--(x42);
			\draw (x42)--(x52);
			\draw [line width=1.2pt] (x04)--(x22);
			\draw (x22)--(x32);
			\draw (x41)--(x51);
			\draw (x14)--(x24)--(x33);
			\draw (x43)--(x53);
			
			\draw (3.25,9)--(v1)--(3.75,9);
			\node at (3.25,9.2) {$4$};
			\node at (3.75,9.2) {$5$};
			
			\draw (5,9)--(v2);
			\node at (5,9.2) {$3$};
			
			\draw (6.25,9)--(v3)--(6.75,9);
			\node at (6.25,9.2) {$2$};
			\node at (6.75,9.2) {$5$};
			
			\draw (7.75,9)--(v4)--(8.25,9);
			\node at (7.75,9.2) {$3$};
			\node at (8.25,9.2) {$6$};
			
			\draw (9.25,9)--(v5)--(9.75,9);
			\node at (9.25,9.2) {$3$};
			\node at (9.75,9.2) {$4$};

			\node at (2,1.5) {$X_{i-2}$};
			\node at (3.5,1.5) {$X_{i-1}$};
			\node at (5,1.5) {$X_{i}$};
			\node at (6.5,1.5) {$X_{i+1}$};
			\node at (8,1.5) {$X_{i+2}$};
			\node at (9.5,1.5) {$X_{i+3}$};
			
			\node at (1,8.25) {$v_0$};
			\node at (2,8.25) {$v_{i-2}$};
			\node at (3.5,8.25) {$v_{i-1}$};
			\node at (5,8.25) {$v_{i}$};
			\node at (6.5,8.25) {$v_{i+1}$};
			\node at (8,8.25) {$v_{i+2}$};
			\node at (9.5,8.25) {$v_{i+3}$};
			\node at (10.5,8.25) {$v_{k-1}$};
			
			\node at (1.2,4.5) {$\cdots$};
			\node at (10.3,4.5) {$\cdots$};

			\node at (10.3,6) {$H$};
			\node at (0.7,9) {$C$};
		\end{tikzpicture}
		\caption{Illustration of $H$ with $L_C(v_{i-2})=\{1,3,4,5\}$, $L_C(v_{i-1})=\{1,2,3,6\}$, $L_C(v_{i})=\{3,4,5,6\}$, $L_C(v_{i+1})=\{3,4,6\}$, $L_C(v_{i+2})=\{1,2,5\}$, and $L_C(v_{i+3})=\{1,2,5,6\}$.
        The edges in $E_0$ are ignored here, thin lines are the short edges and thick lines represent the long edges.}
		\label{fig-H}
	\end{figure}
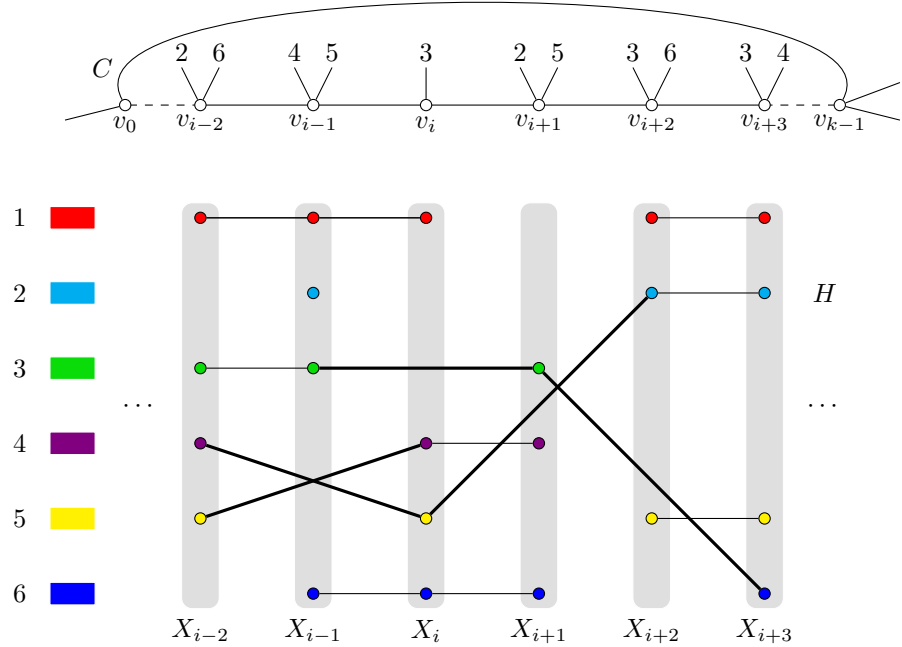 
	
The following observation follows immediately from the definition of $H$.
	
\begin{observation}\label{obs-longedges}
	For each $i\in[k]$, there are at most two long edges between $X_{i-1}$ and $X_{i+1}$.
	Furthermore, if there are two long edges between $X_{i-1}$ and $X_{i+1}$, then they must be $(v_{i-1},\alpha)(v_{i+1},\beta)$ and $(v_{i-1},\beta)(v_{i+1},\alpha)$ for some distinct $\alpha$ and $\beta$ from $L_C(v_{i-1})\cap L_C(v_{i+1})$.
\end{observation} 
	
\begin{claim}\label{claim-size}
	$H$ has no independent set of size $|V(C)|=k$.
\end{claim}
\begin{proof}
	Suppose $I$ is an independent set of $H$ with $|I|=k$. We extend $\phi$ to $G$ so that for each $(v,\alpha)\in I$, $\phi(v)=\alpha$. As $|I|=k$ and $I$ contains at most one vertex of $X_i$ for each $i\in [k]$, all the vertices of $C$ are colored in this manner. Clearly, $\phi$ is proper and $\mathcal{U}_{\phi}(u,G) \neq \emptyset$ for each $u \in  V(C)$. By the definition of $L_C$, $\mathcal{U}_{\phi}(u,G) \neq \emptyset$ for each $u \in  V(G)\setminus V(C)$, which implies that $\phi$ is a proper conflict-free $L$-coloring of $G$, a contradiction.
\end{proof}
	
\begin{claim}\label{claim-oneFoneB}
	For every $i\in [k]$ and every $\alpha \in L_C(v_i)$, there are at most one forward edge and at most one backward edge from $(v_i,\alpha)$.
\end{claim}
	
\begin{proof}
  Suppose that there is a forward short edge from $(v_i,\alpha)$, thus $\alpha\in L_C(v_{i+1})$.
	Then $\alpha \notin \phi(N_{G}(v_{i+1}))$, and hence there is no long forward edges incident with $(v_i,\alpha)$ by definition. So there is at most one forward edge from  $(v_i,\alpha)$. By symmetry, the statement holds for backward edges.
\end{proof}
	
An \emph{extendable structure} $(U,I)$ of $H$ is a pair of a set $U\subseteq X_i\cup X_{i+t}$ for some $i\in [k]$, $t\in \{1,2\}$ and an independent set $I$ of $H$ satisfying the followings:
\begin{enumerate}[label={(C\arabic*)},leftmargin=1.3cm]
	\item \label{C1} $|U \cap X_i| \geq 3$ and  $|U \cap X_{i+t}|=2$.
	\item \label{C2} No backward long edges or no backward short edges incident with vertices from $U\cap X_{i+t}$.
	\item \label{C3} If $t=1$, then $I$ consists of one vertex in $X_{i+2}$ and one in $X_{i+3}$. And if $t=2$, then $I$ consists of a vertex in each of $X_{i+1}$, $X_{i+3}$ and $X_{i+4}$.
	\item \label{C4} No backward edges from vertices in $I$ to $U\cup X_{i-1}$.
\end{enumerate}
Note that an extendable structure is a local structure contained in $\bigcup_{i\leq j\leq i+4} X_j$ for some $i\in [k]$, and we know $k\geq 5$ by Claim \ref{cl:4-cycle}.
If $H$ contains an extendable structure, then it can be extended to an independent set of $H$ with $k$ vertices as in the proof of next claim.
	
\begin{claim}\label{cl:extendablestructure}
	$H$ contains no extendable structure.
\end{claim}
	
\begin{proof} 
	Suppose to the contrary, $(U,I)$ is an extendable structure of $H$.
	Without loss of generality, we may assume that $U\subseteq X_0\cup X_t$ for some $t\in \{1,2\}$. 
  Since every vertex of $H$ has at most one forward edge by Claim \ref{claim-oneFoneB} and $|L_C(v_i)|\geq 3$ for each $i\in [k]$,
  we can greedily add vertices to $I$ from each of $X_{t+3}, X_{t+4}, \ldots, X_{k-1}, X_0$ to obtain an independent set $J^-$ of size $k-1$.
  Observe that $J^- \cap X_t=\emptyset$.  By condition \ref{C2}, at least one of two vertices in $X_{t} \cap U$ has no neighbors in $I \cap (X_{t-1} \cup X_{t-2})$, so we add one such vertex to $J^-$ to have a set $J$ of vertices. By \ref{C4}, $J$ is still an independent set  of $H$ and $|J|=k$, a contradiction to Claim \ref{claim-size}.
\end{proof}

The following observation is frequently used in the the sequel.	
\begin{observation}\label{obs-next two}
	Suppose $\alpha,\beta \in L_C(v_i)$ for some $i \in [k]$ and $\alpha \neq \beta$. Then there exists a color $\gamma \in L_C(v_{i+1})\setminus \{\alpha,\beta\}$  and a color $\tau \in L_C(v_{i+2})\setminus\{\gamma\}$ such that $(v_{i+2},\tau)$ is adjacent to neither $(v_i,\alpha)$ nor $(v_i,\beta)$. (Possibility $\tau \in \{\alpha, \beta\}$ is not excluded.)
\end{observation}	
\begin{proof}
  Suppose that one of $(v_i,\alpha)$ and $(v_i,\beta)$, say $(v_i,\alpha)$, has a neighbor in $X_{i+2}$. As $|L_C(v_{i+2})|\geq 3$, there exists a color $\tau\in L_C(v_{i+2})$ such that $(v_{i+2},\tau)$ is adjacent to neither $(v_i,\alpha)$ nor $(v_i,\beta)$. By Claim \ref{claim-oneFoneB}, we know $\alpha\notin L_C(v_{i+1})$, so we can choose a color $\gamma$ from $L_C(v_{i+1})\setminus \{\alpha,\beta,\tau\}$, as desired.       
  Hence we may assume that both $(v_i,\alpha)$ and $(v_i,\beta)$ have no neighbors in $X_{i+2}$, then we can arbitrarily choose a color $\gamma \in L_C(v_{i+1})\setminus\{\alpha, \beta\}$ and a color $\tau \in L_C(v_{i+2}) \setminus \{\gamma\}$ as desired.
\end{proof}
 
\begin{claim}\label{claim-2-regualr}
	Every vertex of $H$ is incident with exactly one forward edge and one backward edge.
\end{claim}
	
\begin{proof}
	Suppose to the contrary, $(v_0,\alpha)$ has no backward edge for some $\alpha \in L_C(v_0)$. Argument for forward edge is symmetric.
		
	First assume that there exists a color $\beta \in L_C(v_1)$ such that $\beta \neq \alpha$ and $(v_1,\beta)$ has no neighbors in $X_{k-1}$. As $|L_C(v_{k-1})| \geq 3$, by Pigeonhole Principle, there exists two colors, say $1,2 \in L_C(v_{k-1})$, such that backward long edges or backward short edges are missed from $(v_{k-1},1), (v_{k-1}, 2)\in X_{k-1}$. Then $U=X_{k-2}\cup\{(v_{k-1},1), (v_{k-1},2)\}$ and $I=\{(v_0,\alpha), (v_1,\beta)\}$ form an extendable structure with $t=1$: The conditions \ref{C1} and \ref{C3} are obvious. The condition \ref{C2} follows from the choice of $(v_{k-1},1)$ and $(v_{k-1},2)$, and the condition \ref{C4} is satisfied since $(v_0, \alpha)$ has no backward edge and $(v_1,\beta)$ has no backward edge to $X_{k-1}$, a contradiction.
		
	Thus by the definition of $H$, without loss of generality, we deduce that there exist two distinct colors, say $1,2$, such that $L_C(v_1)=\{\alpha,1,2\}$, $\phi(N_{G'}(v_{1}))=\{1,2\}$ and $\{1,2\} \subsetneq L_C(v_{k-1})$. By Observation \ref{obs-next two}, there exists a color $\gamma \in L_C(v_2)\setminus \{1,2\}$ and $\tau \in L_C(v_3) \setminus \{\gamma\}$ such that $(v_3,\tau)$ has no neighbors in $\{(v_1,1),(v_1,2)\}$. Then $U=X_{k-1}\cup \{(v_1,1), (v_1,2)\}$ and $I=\{(v_0,\alpha), (v_2,\gamma), (v_3,\tau)\}$ form an extendable structure with $t=2$: The conditions \ref{C1}, \ref{C2} and \ref{C3} are obvious. The condition  \ref{C4} is satisfied since $(v_0, \alpha)$ has no backward edge, and both $(v_2,\gamma)$ and $(v_3,\tau)$ have no backward edge to $\{(v_1,1),(v_1,2)\}$, a contradiction.
\end{proof} 
	
\begin{claim}\label{claim-at most one short}
	For every $i\in [k]$, there is at most one forward short edge incident with vertices of $X_i$.
\end{claim}
\begin{proof}
	Without loss of generality assume $\{1,2,3\}\subseteq L_C(v_0)$, both $(v_0,1)$ and $(v_0,2)$ incident with forward short edges.
	Let $U=\{(v_0,1), (v_0,2), (v_0,3), (v_1,1), (v_1,2)\}$.
	If $(v_0,3)$ is incident with forward short edge, then by Observation \ref{obs-next two}, there exist colors $\gamma\in L_C(v_2)\setminus\{1,2\}$ and $\tau\in L_C(v_3)\setminus\{\gamma\}$ such that neither $(v_1,1)$ nor $(v_1,2)$ is adjacent to $(v_3,\tau)$.
	So $U$ and $I=\{(v_2,\gamma), (v_3,\tau)\}$ form an extendable structure with $t=1$, a contradiction.
	Hence $(v_0,3)$ is adjacent to $(v_2,\alpha)$ for some $\alpha\in L_C(v_2)\setminus\{1,2\}$.
		
		
	First assume that there exists a color $\beta \in L_C(v_2) \setminus \{1,2,\alpha\}$. Clearly, $(v_2,\beta)$ has no neighbors in $\{(v_0,1),(v_0,2),(v_0,3)\}$. If there exists a color $\gamma \in L_C(v_3) \setminus \{\beta\}$ such that $(v_3,\gamma)$ has no neighbors in $\{(v_1,1),(v_1,2)\}$, then $U$ and $I=\{(v_2,\beta), (v_3,\gamma)\}$ form an extendable structure with $t=1$, a contradiction. 
  Thus we deduce that $L_C(v_3)=\{1,2,\beta\}$ and $\phi(N_{G'}(v_2))=\{1,2\}$, so $1,2\notin L_C(v_2)$. As there is a color $\beta'\in L_C(v_2)\setminus\{1,2,\alpha,\beta\}$, $U$ and $I=\{(v_2,\beta'), (v_3,\beta)\}$ form an extendable structure with $t=1$, a contradiction.
		
	So we have $L_C(v_2)=\{1,2,\alpha\}$.
	Using the same argument repeatedly, we conclude that $\{1,2\}\subseteq L_C(v_i)$ for every $i\in [k]$.
	Then $I=\{(v_{2m+1},1)\mid 0\leq m\leq \left\lfloor (k-1)/2 \right\rfloor\}\cup \{(v_{2m},2)\mid 0\leq m\leq \left\lfloor (k-1)/2 \right\rfloor\} \cup \{(v_0,3)\}$ is an independent set of $H$ with $|I|=k$, a contradiction.
\end{proof}
	
If $|X_i| \geq 4$ for some $i \in [k]$, then by Claims~\ref{claim-2-regualr} and \ref{claim-at most one short}, there are at least three long edges between $X_{i}$ and $X_{i+2}$, a contradiction to Observation \ref{obs-longedges}. Thus $|X_i|=3$ for each $i \in [k]$, two of which incident with forward long edges and the rest one incident with a forward short edge.
	
If all the short edges form a matching of $H$, then we can easily find an independent set $I$ of $H$ with $|I|=k$ by choosing vertices incident with forward short edges, a contradiction to Claim \ref{claim-size}. Therefore, without loss of generality, there must exist a color, say $\alpha$, which appears in $L_C(v_0) \cap L_C(v_1) \cap L_C(v_2)$. 
Then there are two distinct colors in $L_C(v_0)\setminus\{\alpha\}$, say 1 and 2, and there are two distinct colors in $L_C(v_1)\setminus\{\alpha,1,2\}$, say 3 and 4 such that all of $\{(v_0,1),(v_0,2),(v_1,3),(v_1,4)\}$ are incident with forward long edges. By Observation \ref{obs-longedges} and the fact $|L_C(v_i)|=3$ for every $i\in [k]$, we know that $L_C(v_0)=L_C(v_2)=\{\alpha,1,2\}$ and $L_C(v_1)=\{\alpha,3,4\}$, $L_C(v_3)=\{\beta, 3,4\}$ for some color $\beta\notin \{3,4\}$.
If $\beta\neq \alpha$, then 
$U=X_0 \cup \{(v_1,3),(v_1,4)\}$ and $I=\{(v_2,\alpha),(v_3,\beta)\}$ form an extendable structure with $t=1$: the conditions \ref{C1}, \ref{C2} and \ref{C3} are obvious by the choice of vertices, and the condition \ref{C4} is satisfied since $(v_2,\alpha)$ has a backward short edge to $(v_1,\alpha)\notin U$ and the long edges between $X_1$ and $X_3$ are $(v_1,3)(v_3,4)$ and $(v_1,4)(v_3,3)$, which are not incident to $(v_3,\beta)$, a contradiction.
    
Thus $L_C(v_3)=\{\alpha, 3,4\}$. Using the same argument repeatedly, we conclude that $k$ is even, moreover, $L_C(v_j)=\{\alpha,1,2\}$ if $j$ is even, and $L_C(v_j)=\{\alpha,3,4\}$ if $j$ is odd. Then $I=\{(v_j,1)\mid  j \in [k] \text{ and } j \text{ is even}\} \cup \{(v_j,3)\mid  j \in [k] \text{ and } j \text{ is odd}\}$ is an independent set of $H$ with $|I|=k$, a contradiction. This completes the proof of Theorem \ref{thm:main}.

\section*{Acknowledgement}
This work is based on the discussion at Zhejiang Normal University, China. 
We appreciate Xuding Zhu for inviting the first and the second authors there.
M. Kashima has been supported by Keio University SPRING scholarship JPMJSP2123 and JST ERATO JPMJER2301.
R. \v{S}krekovski has been partially supported by the Slovenian Research Agency and Innovation (ARIS) program P1-0383, project J1-3002, and the annual work program of Rudolfovo. 
R. Xu has been supported by National Science Foundation for Young
Scientists of China, grant number: 12401472.


\begin{thebibliography}{99}
	\bibitem{Bondy}
	J.~A.~Bondy and U.~S.~R.~Murty, Graph Theory, Graduate Texts in Mathematics, vol. 244, Springer, New York, 2008.
	\bibitem{Borodin}
	O.~V.~Borodin, Criterion of chromaticity of a degree prescription (in Russian), \textit{Abstracts of IV All-Union Conf. on Th. Cybernetics}, 1977, 127--128.
	\bibitem{Borodin1979}
	O.~V.~Borodin. Problems of colouring and of covering the vertex set of a graph by induced subgraphs. \textit{Ph.D. Thesis, Novosibirsk State University, Novosibirsk} (1979) (in Russian).
	\bibitem{Caro1}
	Y.~Caro, M.~Petru\v{s}evski, and R.~\v{S}krekovski, Remarks on odd colorings of graphs,
	\textit{Discrete Appl. Math.} 321 (2022), 392--401.
	\bibitem{Caro2}
	Y.~Caro, M.~Petru\v{s}evski, and R. \v{S}krekovski, Remarks on proper conflict-free colorings of graphs,
	\textit{Discrete Math.} 346 (2023), 113221. 
	\bibitem{CCKP2023}
	E.-K. Cho, I. Choi, H. Kwon, and B. Park, Brooks-type theorems for relaxations of square colorings, 
  \textit{Discrete Math.} 348 (2025), 114233.
	\bibitem{Cranston}
	D.~W.~Cranston, C.~H.~Liu, Proper conflict-free coloring of graphs with large maximum degree, \textit{SIAM J. Discrete Math.} 38 (4) (2024), 3004--3027.
	\bibitem{DOP2023}
	T.~Dai, Q.~Ouyang, and F.~Pirot, New bounds for odd colorings of graphs, \textit{Electron. J. Combin.} 31 (4) (2024), Article No. P4.57.
	\bibitem{Dvorak} 
	Z.~Dvo\v{r}\'{a}k and L.~Postle, Correspondence coloring and its application to list-coloring planar graphs without cycles of length 4 to 8,
	\textit{J. Combin. Theory Ser. B} 129 (2018), 38--54.
	\bibitem{Even}
	G.~Even, A.~Lotker, D.~Ron, and S.~Smorodinsky, Conflict-free colorings of simple geometric regions with application to frequency assignment in cellular networks, 
	\textit{SIAM J. Comput.} 33 (2003), 94--136.
	\bibitem{Erdos}
	P.~Erd\H{o}s, A.~L.~Rubin, and H.~Taylor, Choosability in graphs,
	\textit{Proc. West Coast Conf. on Combinatorics, Graph Theory and Computing} (Humboldt State Univ., Arcata, Calif., 1979). 
	\bibitem{Fabrici}
	I.~Fabrici, B.~Lu\v{z}ar, S.~Rindo\v{s}ov\'{a}, and R.~Sot\'{a}k, Proper conflict-free and unique-maximum colorings of planar graphs with respect to neighborhoods,
	\textit{Discrete Appl. Math.} 324 (2023), 80--92.
	\bibitem{Kamyczura}
	M.~Kamyczura and J.~Przybylo, On conflict-free proper colourings of graphs without small degree vertices,
	\textit{Discrete Math.} 347 (2024), 113712.
	\bibitem{LY2013}
	C.-H.~Liu and G.~Yu, Linear colorings of subcubic graphs, \textit{European J. Combin.} 34 (2013), 1040--1050.
	\bibitem{Liu2024}
	C.-H.~Liu, Proper conflict-free list-coloring, odd minors subdivisions, and layered treewidth,
	\textit{Discrete Math.} 347 (2024), 113668.
  \bibitem{LRarxiv}
  C.-H.~Liu and B.~Reed, Asymptotically optimal proper conflict-free coloring, \textit{Random Structures \& Algorithm}, 66 (3) (2025), e21285.
	\bibitem{Petrusevski}
	M.~Petru\v{s}evski and R.~\v{S}krekovski, Colorings with neighborhood parity condition,
	\textit{Discrete Appl. Math.} 321 (2022), 385--391.
	\bibitem{Thomassen}
	C.~Thomassen. Color-critical graphs on a fixed surface. \textit{J. Combin. Theory Ser. B}, 70(1):67--100, 1997.
  \bibitem{Vizing}
  V. G. Vizing. Coloring the graph vertices with some prescribed colors (in Russian). \textit{Metody Diskret. Anal. v. Teorii Kodov i Shem} 101 (1976), 3--10.
\end{thebibliography}
\end{document}